\documentclass{ws-procs9x6-arXiv-2}

\usepackage{amscd}
\usepackage{array}
\DeclareMathOperator{\E}{\mathbb{E}} 
\DeclareMathOperator{\Ideal}{Ideal}
\DeclareMathOperator{\Prob}{\mathbb{P}}
\DeclareMathOperator{\RowSpan}{RowSpan}
\DeclareMathOperator{\Span}{Span}
\DeclareMathOperator{\XOR}{XOR}

\DeclareMathOperator{\out}{out}
\newcommand{\arcs}{\mathcal A}
\newcommand{\cocoa}{{\hbox{\rm C\kern-.13em o\kern-.07em C\kern-.13em o\kern-.15em A}}}
\newcommand{\complex}{\mathbb{C}}
\newcommand{\cuts}{\mathcal S}
\newcommand{\cycles}{\mathcal C}
\newcommand{\design}{\mathcal D}
\newcommand{\divergence}[2]{D\left(#1\Vert#2\right)}
\newcommand{\edges}{\mathcal E}
\newcommand{\euler}{\mathrm{e}}
\newcommand{\expectat}[2]{{\E}_{#1}\left[#2\right]}
\newcommand{\expof}[1]{\exp\left(#1\right)}
\newcommand{\idealof}[1]{\Ideal\left(#1\right)}
\newcommand{\integers}{\mathbb{Z}}
\newcommand{\logof}[1]{\log\left(#1\right)}
\newcommand{\loops}{\mathcal L}
\newcommand{\probat}[2]{\Prob_{#1}\left(#2\right)}
\newcommand{\probof}[1]{\Prob\left(#1\right)}
\newcommand{\rationals}{\mathbb{Q}}
\newcommand{\reals}{\mathbb{R}}
\newcommand{\setof}[2]{\left\{#1 \colon #2 \right\}}
\newcommand{\set}[1]{\left\{#1\right\}}
\newcommand{\spanof}[1]{\Span\left(#1\right)}
\newcommand{\sspace}[1]{\mathcal #1}
\makeatletter
\newif\if@borderstar
\def\bordermatrix{\@ifnextchar*{%
\@borderstartrue\@bordermatrix@i}{\@borderstarfalse\@bordermatrix@i*}%
}
\def\@bordermatrix@i*{\@ifnextchar[{\@bordermatrix@ii}{\@bordermatrix@ii[() ]}}
\def\@bordermatrix@ii[#1]#2{%
\begingroup
\m@th\@tempdima8.75\p@\setbox\z@\vbox{%
\def\cr{\crcr\noalign{\kern 2\p@\global\let\cr\endline }}%
\ialign {$##$\hfil\kern 2\p@\kern\@tempdima & \thinspace %
\hfil $##$\hfil && \quad\hfil $##$\hfil\crcr\omit\strut %
\hfil\crcr\noalign{\kern -\baselineskip}#2\crcr\omit %
\strut\cr}}%
\setbox\tw@\vbox{\unvcopy\z@\global\setbox\@ne\lastbox}%
\setbox\tw@\hbox{\unhbox\@ne\unskip\global\setbox\@ne\lastbox}%
\setbox\tw@\hbox{%
$\kern\wd\@ne\kern -\@tempdima\left\@firstoftwo#1%
\if@borderstar\kern2pt\else\kern -\wd\@ne\fi%
\global\setbox\@ne\vbox{\box\@ne\if@borderstar\else\kern 2\p@\fi}%
 \vcenter{\if@borderstar\else\kern -\ht\@ne\fi%
 \unvbox\z@\kern-\if@borderstar2\fi\baselineskip}%
 \if@borderstar\kern-2\@tempdima\kern2\p@\else\,\fi\right\@secondoftwo#1 $%
 }\null \;\vbox{\kern\ht\@ne\box\tw@}%
 \endgroup}
 \makeatother 
 
\begin{document}

 \title{Toric Statistical Models:\\ Ising and Markov }

 \author{Giovanni Pistone}

 \address{Collegio Carlo Alberto\\ E-mail: giovanni.pistone@carloalberto.org\\
 Home page: \url{http://www.giannidiorestino.it}
 }

 \author{Maria Piera Rogantin}

 \address{DIMA Universit\`a di Genova\\ Genova, Italy\\ E-mail: rogantin@dima.unige.it\\
 Home page: \url{http://www.dima.unige.it/~rogantin}
 }

 \begin{abstract}{\bf Abstract} This is a review of current research in Markov chains as toric statistical models. Its content is a mixture of background information, results from the relevant recent literature, new results, and work in progress.
 \end{abstract}


 \bodymatter

 \section{Introduction}
 We discuss a selection of topics in Algebraic Statistics, mainly about Ising models and Markov models. Our presentation of the basics is slightly different from other excellent presentations of the topic and it is based on work in progress and on previous conference presentations, in particular our presentations at the Second CREST-SBM International Conference \emph{Harmony of Gr\"obner Bases and the Modern Industrial Society}, June 28 - July 2, 2010, Osaka, Japan. Due to the review character of this paper, we do not have in-line references, but we give commented references in Bibliographical Notes at the end of each section.
 \section{Lattice exponential families}
 Our introductory example is the \emph{Ising model} from Statistical Physic. Given an undirected graph without loops $(V,\edges)$ we consider a collection $X_v$ of $\pm 1$-valued random variables on the finite sample space $(\sspace X,\mu)$. For each edge $\overline{vw} = e \in \edges$, the $\pm 1$-valued random variable $X_e=X_vX_w$ is called an \emph{interaction}. The exponential family of densities
 \begin{equation} \label{eq:Ising}
   p_\theta = \expof{\sum_{v\in V} \theta_v X_v + \sum_{e \in \edges} \theta_e X_e - \psi(\theta)}, \quad \theta = (\theta_V,\theta_\edges) \in \reals^V \times \reals^\edges, 
 \end{equation}
 is the \emph{Ising model}. The densities are taken with respect to the reference measure $\mu$, hence
 \begin{equation*}
   \psi(u) = \logof{\int_{\sspace X} \expof{\sum_{v\in V} \theta_v X_v + \sum_{e \in \edges} \theta_e X_e} d\mu} \ .
 \end{equation*}
 It is possible to describe the Ising model in a differentiable manifold, that is without reference to any specific chart, by saying that \eref{eq:Ising} is a special parameterization of the set of all strictly positive probability densities $p$ such that
 \begin{equation}\label{eq:EF}
   \log p \in \mathcal V = \spanof{1; X_v \colon v\in V; X_e \colon e \in \edges}.
 \end{equation}

 However, the Ising model has an extra special feature, namely the so-called \emph{canonical statistics}, i.e. the linear basis of $\mathcal V$ which is used to obtain the parameterization in \eref{eq:Ising}, are \emph{integer valued} random variables. We call a model of this type a \emph{lattice exponential family} LEF. It is always possible to parameterize a LEF with nonnegative and non strictly positive canonical statistics. For example, in the Ising model we can use the binary variables $A_v = (1-X_v)/2$, $v\in V$, and $A_{\overline{vw}} = A_v \XOR A_w = A_v + A_w -A_vA_w$, $\overline{vw}\in\edges$, to get the same model in a different parameterization:
 \begin{equation} \label{eq:GB}
   p_\beta = \expof{\sum_{v\in V} \beta_v A_v + \sum_{e \in \edges} \beta_e A_e - \bar\psi(\beta)} \ ,
   \end{equation}
 with the obvious change of parameters $\theta \to \beta$, $\psi \to \bar\psi$. It should be noted that $X_v = (-1)^{A_v}$ and $X_e = X_vX_w = (-1)^{A_v+A_w} = (-1)^{A_e}$. In fact, the re-coding is actually the character group $\integers_2 \ni a \mapsto (-1)^a \in \set{+1,-1} \in \complex$.

 In Statistical Physics, a model as in \eref{eq:GB} is called Gibbs (or Boltzmann-Gibbs) model. The interest of nonnegative but nonpositive canonical statistics appears in the discussion of the limit case where some of the $\beta$'s tend to $-\infty$. In such a case a limit distribution with smaller support is obtained.

 Another parameterization of interest is obtained, by taking in \eref{eq:GB} the nonlinear transformation $t_v=\euler^{\beta_v}$, $v\in V$, $t_e= \euler^{\beta_e}$, $e \in \edges$, to get the \emph{monomial form}
 \begin{equation}\label{eq:LEF1}
   p_t \propto \prod_{v\in V} t_v^{A_v} \prod_{e \in \edges} t_e^{A_e}.
 \end{equation}

 Let us make a second remark. The random variable $\log p$ belongs to the vector space $\mathcal V$ generated by the constants and the canonical statistics if, and only if, it is orthogonal in $\reals^{\sspace X}$ to each random variable $K$ in the orthogonal space $\mathcal V^\perp$. In other words, a density $p$ belongs to the model of \eref{eq:GB} for some $\beta$, or to the model in \eref{eq:LEF1} for some $t$, if, and only if, the equation
 \begin{equation} \label{eq:LEF2-}
   0 = \sum_{x \in \sspace X} \log p(x) K(x) = \logof{\prod_{x\in\sspace X} p(x)^{K(x)}} \ ,
 \end{equation}
 holds for all $K$ such that
 \begin{equation*}
   \sum_{x\in\sspace X} A_v(x)K(x)=0, v\in V,\quad \sum_{x\in\sspace X} A_e(x)K(x)=0, e\in\edges \ .
 \end{equation*}

 By considering the positive and negative part, $K=K_+-K_-$, \eref{eq:LEF2-} can be written as
 \begin{equation}\label{eq:LEF2}
   \prod_{x\in\sspace X} p(x)^{K_+(x)} = \prod_{x\in\sspace X} p(x)^{K_-(x)}.
 \end{equation}
 This argument is true for all exponential families. In particular, in the lattice case, it is possible to find a vector basis of the orthogonal space whose elements $K$ are all integer valued. As a consequence, \eref{eq:LEF1} and \eref{eq:LEF2} are both polynomials with indeterminates $p(x)$, $x\in\sspace X$, $t_v$, $v\in V$, and $t_e$, $e\in\edges$. The binomials in \eref{eq:LEF2} are the \emph{polynomial invariants} of the LEF model.

 In the Ising model it is easy to find a linear basis of the orthogonal space, namely the set $\mathcal J$ of all interactions $X_J = \prod_{v \in J} X_v$, $J \subseteq V$, which are \emph{not} included in the model itself.

 We turn now to the study of statistical models of the special monomial type of \eref{eq:LEF1}. Many cases could support this approach, but in our view the basic one is the following: in \eref{eq:EF} the probability $p$ is assumed to be strictly positive, while both \eref{eq:LEF1} and \eref{eq:LEF2} make sense when $p(x)=0$ at some $x\in \sspace X$.
 \subsection*{Notes} The Ising model is named after the physicist Ernst Ising (1900-1998) and is the basic mathematical model for ferromagnetism. We do not discuss at all its applications to Statistical Physics, where in fact special cases are considered, see e.g. \cite{gallavotti:1999SM}. The unifying concept of exponential family was fully developed in the classical monograph by Barndorff-Nielsen\cite{barndorff-nielsen:78}; a recent exposition of its multiple applications is the review by Wainwright and Jordan\cite{wainwright|jordan:2008}. The importance of a parameter free and geometrical approach was discovered by Cen\c cov\cite{cencov:72} and evolved into what is now called nonparametric Information Geometry, see the seminal papers by Phil Dawid\cite{dawid:75,dawid:1977AS} and the functional version by Pistone and Sempi\cite{pistone|sempi:95}. The algebraic approach emerged in the 90's. It was first outlined in a monograph by Pistone, Riccomagno and Wynn\cite{pistone|riccomagno|wynn:2001} and fully developed in a paper by Geiger, Meek and Sturmfels\cite{geiger|meek|sturmfels:2006}. Currently there is an extensive literature---tagged Algebraic Statistics---we will refer to in the following sections. 
 \section{$A$-model}
 We work on a finite sample space $\sspace X$ with reference measure $\mu$. We consider an nonnegative integer \emph{model matrix} $A \in \integers_{\ge}^{m+1,\sspace X}$ representing $m+1$ random variables $A_i$, $i=0,1,\dots,m$. The elements of the matrix $A$ are denoted by $A_i(x)$, $i = 0 \dots m, x \in \sspace X$. We assume the row $A_0$ to be the constant 1. The $x$-column of $A$, say $A(x)$, is a multi-exponent of the \emph{monomial term}
 \begin{equation}
   \label{eq:1}
         t^{A(x)} = t_0 t_1^{A_1(x)} \cdots t_m^{A_m(x)} \ .
 \end{equation}
 \begin{definition}[$A$-model]
 The monomial model of the model matrix $A$ (briefly, the \emph{$A$-model}) is defined as follows.
 \begin{enumerate}
 \item The \emph{unnormalized probability densities} of the $A$-model are of the form 
     \begin{equation*}
       q(x;t) = t^{A(x)}, \quad x \in \sspace X,
     \end{equation*}
 for all $t \in \reals_{\ge}^{m+1}$ \emph{such that $q(\cdot;t)$ is not identically zero}.
 \item The probability densities with respect to $\mu$ in the $A$-model are
 \begin{equation*}
       p(x;t) = q(x;t)/Z(t), \quad Z(t)=\sum_{x \in \sspace X} q(x;t) \mu(x).
 \end{equation*}
 \item If $t > 0$, $\beta = \log t$ and $q(x;\beta) = \expof{\beta \cdot A(x)}$, i.e. the \emph{interior} of the $A$-model is a LEF in the parameters $\beta$.
 \end{enumerate}
 \end{definition}

 The probability density does not depend on $t_0$, so that we usually drop the $t_0$ parameter:
 \begin{equation}
   \label{eq:2}
   p(x;t_1,\dots,t_m) = \frac{t_1^{A_1(x)}\cdot t_m^{A_m(x)}}{\sum_{x\in\sspace X} t_1^{A_1(x)}\cdots t_m^{A_m(x)} \mu(x)} \ .
 \end{equation}
 However, it is useful to keep it in the notation of the unnormalized density which is a \emph{projective} object.

 The product $t_1^{A_1(x)} \cdots t_m^{A_m(x)}$ is strictly positive for $t_i > 0$, $i=1,\dots,m$, and it is identically zero for $t_i=0$ if $A_i(x) > 0$ for all $x\in\sspace X$. If a row $A_i$ is not strictly positive, then the unnormalized density is defined for all $t$ in the face $\set{t_i=0}$ of the positive quadrant $\reals_{\ge}^{m+1}$. This face parameterizes an $A^{i}$-model with all parameters but $t_i$ and sample space $\sspace X^{i} = \setof{x\in\sspace X}{A_i(x)=0}$, $A^{i}$ being the submatrix of $A$ obtained deleting the $i$-th row and all the columns $x$ such that $A_i(x) > 0$. A similar argument applies to the case where $A_i(x)+A_j(x)=0$ for at least one $x\in\sspace X$. 

 Let us discuss the identifiability of the interior of an $A$-model by deriving a \emph{confounding equation}. 
 \begin{proposition} Two parameter's values $s, t \in \reals_{>}^m$ are such that $p_s = p_t$ if, and only if,
   \begin{equation}
     \label{eq:4}
     (\logof{t_i/s_i} \colon i=0,1,\dots,m) \in e_0 + \ker A^T, \quad e_0 = (1,0,\dots,0) \ .
   \end{equation}

 \end{proposition}

 \begin{proof} Denote by $Z$ the normalizing constant. Then $p_t = p_s$ if, and only if,
   \begin{equation*}
     Z(s) t^{A(x)} = Z(t) s^{A(x)},\quad x \in \sspace X,
   \end{equation*}
 hence 
 \begin{equation*}
   \sum_{i=0}^m (\log t_i - \log s_i) A_i(x) = \log Z(t) - \log Z(s), \quad x \in \sspace X.\end{equation*}
 If we define $\delta_i = (\log t_i - \log s_i)/(\log Z(t) - \log Z(s))$, then  $\delta^T A = 1$. As the first column of $A$ is 1, the first vector of the canonical basis satisfies $e_0^TA=1$, so that the confounding equation is \eref{eq:4}.
 \end{proof}

 Let be given two matrices $A \in \integers_{\ge}^{m+1,\sspace X}$ and $B \in \integers_{\ge}^{n+1,\sspace X}$. When the interior of the $A$-model does represent the same statistical model as the interior on the $B$-model?

 \begin{proposition}
   The interiors of the $A$-model and the $B$-model coincide if, and only if, $\RowSpan A = \RowSpan B$.
 \end{proposition}

 \begin{proof} Assume that for $t > 0$ and $s > 0$ there is a positive constant $c$ such that 
 \begin{equation}
   \label{eq:3}
   t^{A(x)} = cs^{B(x)}, \quad x \in \sspace X.
 \end{equation}
 It follows that $\sum_{i=0}^m \log t_i A_i(x) = \log c + \sum_{j=0}^n \log s_j B_j(x)$.
 \end{proof}

 It is relevant to note that the equivalence of the interiors does not imply the equivalence of the borders, as the following example shows. This topic is discussed in the next Section. 

 The simplest example of $A$-model is the Binomial$(n,p)$ with state space $\sspace X = \set{0,1,2,3,\dots,n}$, measure $\mu(x) = \binom n x$, model matrix 
 \begin{equation*}
       A =
       \bordermatrix[{[}{]}]{%
         & 0 & 1 & 2 & 3 & \cdots & n \cr
       0 & 1 & 1 & 1 & 1 & \cdots & 1 \cr 
       1 & 0 & 1 & 2 & 3 & \cdots & n} \ , 
   \end{equation*}
 unnormalized density $q(x;t_0,t_1) = t_0 t_1^x$, and density $p(x;t_1) = {t_1^x}/{(1+t_1)^n}$, $x=0,1,\dots,n$ and $t_1 \ge 0$.

 A second monomial model with the same interior has model matrix
 \begin{equation*}
       B =
       \bordermatrix[{[}{]}]{%
         & 0 & 1 & 2 & \cdots & n-1 & n \cr
       0 & 1 & 1 & 1 & \cdots & 1 & 1 \cr 
       1 & 0 & 1 & 2 & \cdots & n-1 & n \cr
       2 & n & n-1 & n-2 & \cdots & 1 & 0} \ , 
   \end{equation*}
 unnormalized density $q(x;t_0,t_1,t_2) = t_0 t_1^x t_2^{n-x}$, and density $p(x;t_1,t_2) = {t_1^xt_2^{n-x}}/{(t_1+t_2)^n}$, $t_1, t_2 \ge 0$.

 The Gibbs model in \eref{eq:GB} with state space $\sspace X = \set{+1,-1}^V$ has a model matrix whose rows are indexed by $0, V, \edges$ and entries $A_0(x)=1$, $A_v(x) = (1-x_v)/2$, $A_{\overline{vw}}(x) = 3/4-x_v/4-x_w/4-x_vx_w/4$.  

 In some applications the statistical model is further \emph{constrained}. We consider here two types of contrains: linear constrains on the probability densities and linear contrains on the parameters of the monomial model.

 In the first case a matrix $C \in \integers^{k,n}$ is given and the statistical model is $q(x;t) = t^{A(x)}$, restricted to all $t$'s such that $\sum_{x \in \sspace X} C_i(x) q(x;t) = 0$, $i = 1, \dots, k$. In the second case the parameters $t$ are constrained by a linear variety. In general, the constrained statistical model is not anymore an $A$-model. Instead, it is an instance of a \emph{curved exponential family}. 

 \subsection*{Notes} The term $A$-model was first used in the seminal paper by Geiger, Meek and Sturmfels\cite{geiger|meek|sturmfels:2006}. It is currently of general use, but unfortunately the definition has been adapted by various authors to their special needs. For example, the original paper assumes the column sums to be constant, which we do not. A further (small) issue comes from the presentation of the matrix $A$: in the statistical literature the model matrix has sample points as \emph{rows}, while the algebraic literature takes sample points as \emph{columns}. The geometry of curved exponential families was first discussed by Efron\cite{efron:1978}. 

 \section{Toric ideals and the closure of the $A$-model}
 The kernel of the ring homomorphism from $\rationals[q(x) \colon x \in \sspace X]$ to $\rationals[t_0,\dots,t_m]$ defined by $q(x) \mapsto t^{A(x)}$, $x \in \sspace X$,
 is the \emph{toric ideal} of A, I$(A)$. It is a prime ideal generated by binomials
 \begin{equation} \label{eq:allbinomials}
   \prod_{x \colon k(x) > 0} q(x)^{k^+(x)} - \prod_{x \colon k(x) < 0} q(x)^{k^-(x)} \ ,
 \end{equation}
 with $k \in \integers^{\sspace X}\cap \ker A$, hence there exists a finite generating set of binomials. The polynomials in \eref{eq:allbinomials} are the \emph{polynomial invariants} of the $A$-model and all its unnormalized densities belong to the intersection of the variety of the toric ideal with $\reals_{\ge}^{\sspace X}$. Because of the assumption $A_0=1$, we have $\sum_{x \in \sspace X} k(x) = 0$, so that the binomials in \eref{eq:allbinomials} are homogeneous polynomials. Hence all densities $p_t = q_t/Z(t)$ in the $A$-model belong to the ideal generated by the same binomial equations. 

 In fact, more is true.
 \begin{proposition}\label{th:closure}
   The intersection of the $A$-variety with the probability simplex is the closure of the $A$-model. 
 \end{proposition}
 We discuss below a slightly different version of this basic result.

 Let $B$ be a model matrix such that the $A$-model and the $B$-model are equal in the interior of the parameter space. Each row of $B$ belongs to the set $\integers_{\ge}^{\sspace X} \cap \RowSpan A$. This set is closed under vector sum and has a unique minimal generating set, which is called Hilbert basis. Each vector in the Hilbert basis is nonnegative and, because of the minimality, has at least one zero. Let $H$ be a matrix with margins $\set{1,\dots,h} \times \sspace X$, whose rows are the vectors of the Hilbert basis.
 \begin{proposition}
   \begin{enumerate}
   \item The $H$-model is the closure of the $A$-model, i.e. each density in the $H$-model is a limit of a sequence in the $A$-model.  
   \item Setting $t_j=0$ in the $H$-model, we obtain a limit $H^j$-model whose support is $\sspace X_j = \setof{x\in\sspace X}{H_j(x)=0}$.
   \item The $H^j$-model on $\sspace X_j$ is the $H$-model conditioned to $\sspace X_j$.
   \end{enumerate}
 \end{proposition}

 It should be noted that the $H$-model in the previous proposition is possibly non minimal among models with the closure property in Item (1). A basis producing a minimal representation of all limits is called a \emph{circuit basis}. If the Hilbert basis is boolean, then it is also a minimal description of the border.

 When the model is constrained, the admissible limits are obtained by intersecting the constrains with the faces of the nonnegative quadrant. This is discussed in the following examples.

 \subsection{Example: the binomial}
 The integer kernel of $A=
     \begin{bmatrix}
 1 & 1 & 1 & 1 & 1 & 1 \cr 
 0 & 1 & 2 & 3 & 4 & 5 
    \end{bmatrix}$ is $\rationals$-generated by the rows of
    \begin{equation*}
      K=
    \begin{bmatrix}
    1  & -2 & 1 & 0 & 0 & 0\\
    0  & 1 & -2 & 1 & 0 & 0\\
    0  & 0 & 1 & -2 & 1 & 0\\
    0  & 0 & 0 & 1 & -2 & 1\\
    \end{bmatrix} \ ,
    \end{equation*}
 and the corresponding binomials are
    \begin{equation*}
      q(0)q(2)-q(1)^2, q(1)q(3)-q(2)^2, q(2)q(4)-q(3)^2, q(3)q(5)-q(4)^2.
    \end{equation*}
 The Hilbert basis of $\RowSpan A$ is $H =  \begin{bmatrix}
       0 & 1 & 2 & 3 & 4 & 5\\
       5 & 4 & 3 & 2 & 1 & 0
     \end{bmatrix}$  and hence $q(x;t_0,t_1,t_2) = t_0t_1^x t_2^{5-x}$.

 The admissible defective supports for limits are $\set{0}$ and $\set{5}$. Assume we add the constrain $p(0)=p(5)$, i.e. the constrain matrix $(1, 0, 0, 0, 0, -1)$. In monomial form the constrain is $t_1^0 t_2^{5-0} = t_1^5 t_2^{5-5}$, i.e. $t_1=t_2$. This constrain happens to be a binomial, and the constrained model reduces to a single distribution, namely the uniform distribution.

 \subsection{Example: 3 binary identical RVs, no 3-way interaction}
 Consider the sample space $\sspace X = \set{+,-}^3$ and model matrix
 \begin{equation}\label{eq:no123} A = 
   \bordermatrix[{[}{]}]{%
    & \scriptstyle +++ & \scriptstyle -++ & \scriptstyle +-+ & \scriptstyle --+ & \scriptstyle ++- & \scriptstyle -+- & \scriptstyle +-- & \scriptstyle --- \cr
 0  & 1 & 1 & 1 & 1 & 1 & 1 & 1 & 1 \cr
 1  & 0 & 1 & 0 & 1 & 0 & 1 & 0 & 1 \cr
 2  & 0 & 0 & 1 & 1 & 0 & 0 & 1 & 1 \cr
 3  & 0 & 0 & 0 & 0 & 1 & 1 & 1 & 1 \cr
 12 & 0 & 1 & 1 & 0 & 0 & 1 & 1 & 0 \cr
 13 & 0 & 1 & 0 & 1 & 1 & 0 & 1 & 0 \cr
 23 & 0 & 0 & 1 & 1 & 1 & 1 & 0 & 0} \ .
 \end{equation}
 It is a special Ising model on a complete graph on 3 vertices. The orthogonal space is generated by the vector of the 3-way interaction
 \begin{equation*}
  X_1X_2X_3=\bordermatrix[{[}{]}]{%
    & \scriptstyle +++ & \scriptstyle -++ & \scriptstyle +-+ & \scriptstyle --+ & \scriptstyle ++- & \scriptstyle -+- & \scriptstyle +-- & \scriptstyle --- \cr
  & 1 & -1 & -1 & 1 & -1 & 1 & 1 & -1}  \ .
 \end{equation*}

 The Hilbert basis is given by the rows of the matrix
 \begin{equation*} H =
 \bordermatrix[{[}{]}]{%
    & \scriptstyle +++ & \scriptstyle -++ & \scriptstyle +-+ & \scriptstyle --+ & \scriptstyle ++- & \scriptstyle -+- & \scriptstyle +-- & \scriptstyle --- \cr
  1 & 1 & 0 & 0 & 0 & 0 & 0 & 0 & 1 \cr
  2 & 0 & 0 & 0 & 0 & 1 & 0 & 1 & 0 \cr
  3 & 0 & 0 & 1 & 0 & 0 & 0 & 1 & 0 \cr
  4 & 0 & 0 & 0 & 1 & 0 & 0 & 0 & 1 \cr
  5 & 0 & 1 & 0 & 0 & 0 & 0 & 1 & 0 \cr
  6 & 0 & 0 & 0 & 0 & 0 & 1 & 0 & 1 \cr
  7 & 0 & 0 & 0 & 0 & 0 & 0 & 1 & 1 \cr
  8 & 1 & 1 & 0 & 0 & 0 & 0 & 0 & 0 \cr
  9 & 1 & 0 & 1 & 0 & 0 & 0 & 0 & 0 \cr
  10 & 1 & 0 & 0 & 0 & 1 & 0 & 0 & 0 \cr
  11 & 0 & 1 & 0 & 1 & 0 & 0 & 0 & 0 \cr
  12 & 0 & 0 & 1 & 1 & 0 & 0 & 0 & 0 \cr
  13 & 0 & 0 & 0 & 1 & 1 & 0 & 0 & 0 \cr
  14 & 0 & 1 & 0 & 0 & 0 & 1 & 0 & 0 \cr
  15 & 0 & 0 & 1 & 0 & 0 & 1 & 0 & 0 \cr
  16 & 0 & 0 & 0 & 0 & 1 & 1 & 0 & 0} \ .
 \end{equation*}

 The previous matrix was computed with a symbolic software. However, we note that each row is obtained by taking a single 1 in the subset where the value of the 3-way interaction equals 1 and another one in the complementary subset, for a total of $4\times4=16$ rows. A proof of the Hilbert basis property could be based on the minimality of the support of such vectors.

 We denote by $s_1,\dots,s_{16}$ the parameters of the $H$-model. The possible reduced supports of limit distributions of the $A$-model are the intersections of the subsets of 6 zeros in each of 16 rows of $H$. For example, if we set $s_1=0$ in the $H$-model, then the set $\set{+++,---}$ has zero probability and the limit model matrix is obtained by conditioning the $A$-model to the remaining support set $\sspace X_1$,
 \begin{equation*}
   A_1 = \bordermatrix[{[}{]}]{%
    & \scriptstyle -++ & \scriptstyle +-+ & \scriptstyle --+ & \scriptstyle ++- & \scriptstyle -+- & \scriptstyle +-- \cr
 0  & 1 & 1 & 1 & 1 & 1 & 1 \cr
 1  & 1 & 0 & 1 & 0 & 1 & 0 \cr
 2  & 0 & 1 & 1 & 0 & 0 & 1 \cr
 3  & 0 & 0 & 0 & 1 & 1 & 1 \cr
 12 & 1 & 1 & 0 & 0 & 1 & 1 \cr
 13 & 1 & 0 & 1 & 1 & 0 & 1 \cr
 23 & 0 & 1 & 1 & 1 & 1 & 0} \ .
 \end{equation*}
 On the subset $\sspace X_1$ the aliasing relation is $X_1X_2+X_1X_3+  X_2X_3 = -1$, therefore one of the interactions depends on the other two interactions. The submatrix with one interaction's row deleted is non-singular. In conclusion, the limit model is the saturated model, i.e. the full simplex of probabilities on $\sspace X_1$.

 We pass now to the discussion of the constrained model. The equality of the marginal distributions reduces to the constrain matrix
 \begin{equation*} C = 
 \bordermatrix[{[}{]}]{%
    & \scriptstyle +++ & \scriptstyle -++ & \scriptstyle +-+ & \scriptstyle --+ & \scriptstyle ++- & \scriptstyle -+- & \scriptstyle +-- & \scriptstyle --- \cr
 1=2 & 0 & 1 & -1 & 0 & 0 & 1 & -1 & 0 \cr
 1=3 & 0 & 1 & 0 & 1 & -1 & 0 & -1 & 0} \ .
 \end{equation*}
 In terms of the parameters $s_1,\dots,s_{16}$ of the $H$-model the constrains are
 \begin{gather*}
   s_{5}s_{8}s_{11}s_{14} + s_{6}s_{14}s_{15}s_{16} - s_{3}s_{9}s_{12}s_{15} - s_{2}s_{3}s_{5}s_{7} = 0\ , \\  s_{5}s_{8}s_{11}s_{14} + s_{4}s_{11}s_{12}s_{13} - s_{2}s_{10}s_{13}s_{16} - s_{2}s_{3}s_{5}s_{7} = 0 \ .
 \end{gather*}
 The intersection of the previous variety with the necessary condition for a border case, i.e. $s_1 \cdots s_{16}=0$, gives the equations of the constrain on the border.

 \subsection*{Notes} The theory of toric ideals is due to Sturmfels\cite{sturmfels:1996}. We do not discuss here an important topic of this area, namely Markov Bases which were introduced in another seminal paper by Diaconis and Sturmfels\cite{diaconis|sturmfels:98}. The border of an $A$-model is discussed in detail in Kahle's thesis\cite{kahle:2010thesis} together with a generalization to general exponential families due to Rauh, Kahle and Ay\cite{rauh|kahle|ay:09}. Here we have associated the border to a special version of the $A$-model using an Hilbert basis as set of canonical statistics, an idea which is mentioned first in Rapallo's thesis\cite{rapallo:2003thesis}. Proofs are published in Malag\`o and Pistone\cite{malago|pistone:1012.0637}. Hilbert basis computations where done using \texttt{4ti2}\cite{4ti2} and \cocoa\cite{CocoaSystem}. Another important topic we do not discuss here is Birch's theorem, see the exposition by Pachter and Sturmfels\cite{pachter|sturmfels:2005}. The discussion outlined in the Examples is new.

 \section{Differentiation of the normalizing constant}\label{sec:diff}
 A key result in exponential families is the relation of the partial derivatives of the cumulant function with the cumulants of the canonical statistics. In particular, the gradient of the cumulant generating function maps the canonical parameters onto the interior of the convex polytope generated by the values of the canonical statistics. We discuss here a version of this in the case of $A$-models. 

 We call \emph{design} any finite set of real vectors. The image of a LEF under the canonical statistics is the \emph{canonical} LEF. Its support is a \emph{design} $\design \subset \integers^m$. In particular, the canonical version of an $A$-model is supported by the design $\design \in \integers_{\ge}^m$ whose points are the columns of the model matrix. 

 The set of all polynomials which are zero on a design is the \emph{design ideal} I$(\design)$. The canonical $A$-model has the form
 \begin{equation*}
   q(x;t) = \prod_{i=1}^m t_i^{x_i}, \quad x \in \design, \quad t_i \ge 0, \quad j=1,\dots,m,
 \end{equation*}
 with normalizing constant (partition function)
 \begin{equation*}
   Z(t) = \sum_{x\in\design} t^x \mu(x) \ .
 \end{equation*}

 In the Weyl algebra $\complex\langle t_1 \dots t_d, \partial_1 \dots \partial_d\rangle$ we define the operators
 \begin{equation*}
 t_i \partial_i - x_i = \partial_i t_i - (1+x_i), \quad i = 1,\dots,m, \quad x \in \design,\end{equation*}
 where the equality follows from the commutation relation $\partial_i t_i = 1 + t_i \partial_i$. For all $x \in \design$ we have
 \begin{equation*}
   (t_i \partial_i - x_i)\bullet t^x = \partial_i \bullet (t_i t^x) - (1+x_i) t^x = 0,
 \end{equation*}
 so that $t_i\partial_i \bullet t^x = x_i t^x$ and, by iteration, \emph{$(t_i\partial_i)^{\alpha} \bullet t^x = x_i^\alpha t^x$}, $\alpha \in \integers_{\ge}$. 

 The operator $(t_i\partial_i)^\alpha$ applied to the polynomial $Z(t) \in \complex[t_1,\dots,t_m]$ gives
 \begin{equation*}
 (t_i\partial_i)^\alpha \bullet Z(t) = \sum_{x \in \design} (t_i\partial_i)^\alpha \bullet t^x \mu(x)= \sum_{x \in \design} x_i^\alpha t^x \mu(x) \ .
 \end{equation*}
 For $i \ne j$ we have the commutation $(t_i\partial_i)(t_j\partial_j) = (t_j\partial_j)(t_i\partial_i)$, hence
 \begin{equation*}
   \prod_{i=1}^m (t_i\partial_i)^{\alpha_i} \bullet Z(t) = \sum_{x \in \design} \prod_{i=1}^m (t_i\partial_i)^{\alpha_i} \bullet t^x \mu(x)= \sum_{x \in \design} \left(\prod_{i=1}^m x_i^{\alpha_i}\right) t^x \mu(x) \ .
 \end{equation*}

 By dividing by the normalizing constant we obtain he following expression for the moments:
 \begin{equation*}
 Z(t)^{-1}  \prod_{i=1}^d (t_i\partial_i)^{\alpha_i} \bullet Z(t) = Z(t)^{-1}  \sum_{x \in \design} \prod_{i=1}^m (t_i\partial_i)^{\alpha_i} \bullet t^x = \expectat t {X^\alpha} \ .
 \end{equation*}

 From the ring homomorphism 
 \begin{equation*}
 A \colon \left\{
 \begin{array}{ccc}
 \complex[x] &\to &\complex\langle t_1 \dots t_m, \partial_1 \dots \partial_m\rangle, \\
 x_i &\mapsto &t_i\partial_i,
 \end{array} \right.
 \end{equation*}
 we have for each polynomial $f \in \reals(x_1,\dots,x_m)$
 \begin{equation*}
   A(f) \bullet Z(t) = \sum_{x\in\design} f(x) t^x \mu(x) \ .
 \end{equation*}

 As $x\in\design$, the polynomial $f$ is identified up to an element of the design ideal. The quotient ring $\reals(x_1,\dots,x_m)/\text{I}(\design)$ has a linear basis $\setof{x^\alpha}{\alpha\in M}$ of monomials called \emph{monomial basis}, with $N = \# M = \#\design$ elements.

 \begin{proposition}\label{prop:differential}
 \begin{enumerate}
 \item 
   Let $\setof{x^\alpha}{\alpha\in M}$, be a monomial basis for $\design$. Then $Z(t)$ satisfies the following system of $N$ linear non-homogeneous differential equations:
   \begin{equation*}
     A(x^\alpha) \bullet Z(t) = \sum_{x\in\design} x^\alpha t^x, \quad \alpha\in M.
   \end{equation*}
 \item \label{prop:differential2}
   Let $f_a(x)$ be the (reduced) indicator polynomial of $a\in \design$. Then $Z(t)$ satisfies the following system of $N$ linear non-homogeneous differential equations:
   \begin{equation*}
     A(f_a) \bullet Z(t) = t^a, \quad a\in\design.
   \end{equation*}
 \label{item:2}
 \item Let $g(p_a \colon a\in\design)$ be a polynomial in the toric ideal of the monomial homomorphism $p_a \mapsto t^a$. Then
   \begin{equation*}
   g\left(A(f_a(x)) \bullet Z(t) \colon a\in\design\right) = 0.
   \end{equation*}
 \end{enumerate}
 \end{proposition}

 In the previous theorem, if the right end sides in the first two Items are expressed in terms of known moments, the equations are homogeneous, e.g. Item \ref{item:2} becomes
 \begin{equation*}
     A(f_a) \bullet Z(t) = p(a;t) Z(t), \quad a\in\design.
 \end{equation*}
 \subsection{Example: 3 binary variables, no 23-- and 123-interactions}
 The model matrix is the same as in \eref{eq:no123} with the 23-row deleted:
 \begin{equation*} A = 
   \bordermatrix[{[}{]}]{%
    & \scriptstyle +++ & \scriptstyle -++ & \scriptstyle +-+ & \scriptstyle --+ & \scriptstyle ++- & \scriptstyle -+- & \scriptstyle +-- & \scriptstyle --- \cr
 0  & 1 & 1 & 1 & 1 & 1 & 1 & 1 & 1 \cr
 1  & 0 & 1 & 0 & 1 & 0 & 1 & 0 & 1 \cr
 2  & 0 & 0 & 1 & 1 & 0 & 0 & 1 & 1 \cr
 3  & 0 & 0 & 0 & 0 & 1 & 1 & 1 & 1 \cr
 12 & 0 & 1 & 1 & 0 & 0 & 1 & 1 & 0 \cr
 13 & 0 & 1 & 0 & 1 & 1 & 0 & 1 & 0} \ .
 \end{equation*}
 The orthogonal space is generated by the missing interactions $X_1X_3$ and $X_1X_2X_3$. Computations where done with the software {\cocoa}. A monomial basis of the design is
 \begin{equation*}
   1, x_{13}, x_{12}, x_3, x_2, x_1, x_{12}x_{13}, x_2x_{13} \ ,
 \end{equation*}
 and the indicator polynomial of the column $A({\scriptsize -++}Two) = 110011$, is expressed in this monomial basis by
 \begin{equation*}
   f_{-++}(x) = \frac12 x_2x_{13} - \frac12 x_{12}x_{13} - \frac14 x_1 + \frac14 x_3 - \frac14 x_{13} + 1 \ .
 \end{equation*}
 It follows that the differential operator of Proposition \ref{prop:differential}\eqref{prop:differential2} is 
 \begin{multline*}
   A(f_{-++}) =\\ 1/2t_2\partial_2t_5\partial_5 - 1/2t_4\partial_4t_5\partial_5 - 1/4t_1\partial_1 + 1/4t_3\partial_3 - 1/4t_5\partial_5 + 1 \ .
 \end{multline*}

 \subsection*{Notes} Here we use the algebraic theory of design which was presented first by Pistone and Wynn\cite{pistone|wynn:1996} and discussed in detail in the quoted monograph\cite{pistone|riccomagno|wynn:2001}. The practical interest and feasibility of the resulting computations is object of current research.

 \section{Markov chain, toric Markov chain}
 \label{sec:MC}
 We consider in this section an homogeneous irreducible Markov chain $X_k$, $k=0,1,\dots$, with state space $V$, initial probability $\pi_0$, transitions $P_{v\to w}$, $v, w \in V$.  Let $\mathcal A = \setof{v\to w}{P_{v\to w} > 0, v \ne w}$ and $ \loops = \setof{v\to v}{P_{v\to v} > 0, v \in V}$. The transitions in $\loops$ are called \emph{loops}. The directed graph of transitions $(V;\mathcal A\cup \loops)$ is defined by $v \to w \in \mathcal A\cup \loops$ if, and only if, $P_{v\to w} > 0$, $v,w \in V$, and it is connected. Let $\omega$ be a trajectory with positive probability, i.e. a path of the graph of transitions, $\omega=\omega_0 \omega_1 \cdots \omega_n$ with $(\omega_{i-1}\to \omega_{i}) \in \mathcal A\cup \loops$, $i=1,\dots,n$. The set of trajectories with $n$ transitions is denoted by $\Omega_n$. For each $\omega\in\Omega_n$, the transition's count is the integer $V\times V$-matrix with elements $N_{v\to w}(\omega) = \sum_{k=1}^n (X_{k-1}=v,X_k=w)$.

 The joint distribution up to the time $n$ on the sample space $\Omega_n$ is a monomial term in the ring $\rationals[\pi_0(v), v\in V; P_a, a\in \arcs\cup\loops]$,
 \begin{align}
 \probat n \omega &= \pi_0(\omega_0) P_{\omega_0\to \omega_1} \cdots P_{\omega_{n-1}\to \omega_n} \notag \\ &= \prod_{v\in V} \pi_0(v)^{(X_0(\omega)=v)} \prod_{a\in \mathcal A\cup \loops} P_a^{N_a(\omega)} \label{eq:MCjoint} \\ &= \prod_{v\in V} \pi_0(v)^{(X_0(\omega)=v)} \prod_{l\in\loops} P_{l}^{N_{l}(\omega)}\prod_{a\in \mathcal A} P_a^{N_a(\omega)} \notag \ . 
 \end{align}

 The sparse matrix whose rows have indexes in $0, V, \arcs\cup\loops$, whose columns have indexes in $\Omega_n$, such that the column $\omega$ is $1$, $\setof{(X_0(\omega)=v)}{v\in V}$, $\setof{N_a(\omega)}{a \in \mathcal A\cup \loops}$, defines a toric statistical model on $\Omega_n$ called \emph{toric Markov chain} (TMC).  

 The  unnormalized density up to the time $n$ of the toric model is
 \begin{align}
 q_n(\omega; t) &= t_0 t_{\omega_0} t_{\omega_0\to \omega_1} \dots  t_{\omega_{n-1}\to \omega_n} \notag \\ &= t_0 \prod_{v\in V} t_v^{(X_0(\omega)=v)} \prod_{a\in \mathcal A\cup \loops} t_a^{N_a(\omega)} \ .
 \label{eq:qTMC}\end{align}

 For example, if the state space is $V=\set{+1,-1}$, a direct computation shows that the TMC is, up to a linear transformation of the canonical parameters, equal to the constrained model of the type of \eref{eq:Ising}, namely
 \begin{equation*}
   \log q_{\alpha,\beta}= \alpha_0 X_0 +  \alpha \sum_{t=1}^{n-1}X_t + \alpha_n X_n + \sum_{t=1}^{n} \beta_t X_{t-1}X_t \ .
 \end{equation*}
 The constrain is a linear constrain on the canonical parameters $\alpha_t = \alpha$, $t=1,\dots,n-1$. In fact, the constrained model is an exponential family.

 In the TMC the $t$'s parameters are not required to be transition probabilities, therefore the Markov chain model is a submodel of the TMC model given by linear restrictions on the $t$'s. It follows that the MC is not an exponential family, but it is a curved exponential family. 

 More precisely, as it is seen from the product form, the TMC from time 0 to time $n$ is a special Markov process with non-homogeneous transition probabilities. If the transition probabilities are homogeneous, then the TMC is a Markov chain. Moreover, we note that the joint distributions of the TMC does not form a projective system for different $n$'s, in particular classical results on the mixture representation of processes whose sufficient statistics are transition's counts do not apply here.
 \begin{proposition}
  Let us define the vector $S(v)=\sum_{w \colon v\to w \in \mathcal A \cup \loops} t_{v_\to w}$, $v\in V$. The TMC on $\Omega_n$ is a MC if, and only if, it is constant, say $S(v)=S$, $v\in V$.
 \end{proposition}
 \begin{proof}
 We denote by $\probat n \omega$ the probability of a trajectory $\omega\in\Omega_n$. A transition matrix $P$ is obtained by normalizing of the $t_a$'s,
 \begin{equation*}\label{eq:trans}
 P_{v\to w} = \frac{t_{v\to w}}{S(v)}, \quad S(v) = \sum_{w\in V} t_{v\to w}, \quad (v\to w) \in \arcs\cup\loops,
 \end{equation*}
 and $P_{v\to w} = 0$ if $(v\to w)\notin\arcs\cup\loops$. As $q_n(v_0v_1\cdots v_{n-1}v_n;t) =q_{n-1}(v_0v_1\cdots v_{n-1};t)t_{v_{n-1}\to v_n}$, the marginal unnormalized  density  up to time $(n-1)$ is
 \begin{equation*}
 \sum_{v_n \in V} q_n(v_0v_1\cdots v_{n-1}v_n;t) =q_{n-1}(v_0v_1\cdots v_{n-1};t)  S(v_{n-1})\ ,
 \end{equation*}
 and the conditional probabilities are
 \begin{equation*}
 \probat n{X_n=v_n \vert X_{n-1} = v_{n-1} \dots  X_0 = v_0} = \frac{q_{n-1} \ t_{v_{n-1}\to v_n}}{q_{n-1}
 S(v_{n-1})}= P_{v_{n-1}\to v_n} \ .
 \end{equation*}
 The marginal unnormalized  density  up to time $(n-2)$ is
 \begin{equation*}
 \sum_{v_{n-1}}  q_{n-1} S(v_{n-1}) =   q_{n-2} \sum_{v_{n-1}} t_{v_{n-2}\to v_{n-1}}S(v_{n-1}) \ ,
 \end{equation*}
 and the conditional probabilities are given by
 \begin{equation*}
 \probat n{X_{n-1}=v_{n-1} \vert \  X_{n-2}=v_{n-2} \dots  X_0=v_0}=\frac{t_{v_{n-2}\to v_{n-1}}
 S(v_{n-1})}{\sum_{v_{n-1}} t_{v_{n-2}\to v_{n-1}}S(v_{n-1})} \ .
 \end{equation*}
 For generic vertices $v,w \in V$ we have:
 \begin{align*}
 \probat n{X_n=w \vert \  X_{n-1}=v} &= \frac{ t_{v\to w}}{ S(v)},\\ 
 \probat n{X_{n-1}=v \vert \ X_{n-2}=w} &= \frac{t_{v\to w} S(w)}{ \sum_{u} t_{v\to u}S(u)} \ ,
 \end{align*}
 hence TMC is an homogeneous MC if, and only if,
 \begin{equation*}
 S(w) =\frac{ \sum_{u} t_{v\to u}S(u)}{ S(w)} \ ,
 \end{equation*}
 i.e. if and only if $S(v)=S$, $v\in V$. In such a case the $t_{v\to w}$ is proportional to the transition matrix $P_{v\to w}$.
 \end{proof}

 A second description of the relation between the MC and the TMC follows by writing $t_{v\to w}$ as $S(v)P_{v\to w}$ in \eref{eq:qTMC}. The unnormalized density $q_n(\omega; t)$ can be re-written as
 \begin{align*}
 q_n(\omega; t) &= t_0 \prod_{v\in V} t_v^{(X_0(\omega)=v)} \prod_{v\to w \in \mathcal A \cup \loops} S(v)\ P_{v\to
 w}^{N_{v\to w}(\omega)} \\ &=  t_0 \left( \prod_{v\in V} t_v^{(X_0(\omega)=v)} \prod_{v\to w \in
 \mathcal A \cup \loops}P_{v\to w}^{N_{v\to w}(\omega)}\right) \prod_{v\in V} S(v)^{N_{v\to \cdot}(\omega)},
 \end{align*}
 where $N_{v\to \cdot} = \sum_{w\in \out(v)} N_{v\to w}$ is the number of exits from $v$, including transitions from $v$ to $v$ itself. This distribution is the distribution of a MC with the added weight $\prod_{v\in V} S(v)^{N_{v\to \cdot}}$. Or, we can say that the TMC, given the number of outs $N_{v\to\cdot}$, $v\in V$, is a MC. 

 The normalizing constant of a TMC on a graph $(V,\arcs\cup\loops$ is $Z(t) = \sum_{\omega\in\Omega_n} q(\omega;t)$ and the normal equations for the maximum likelihood estimation can be written as polynomial equations with the operator introduced in \sref{sec:diff} as $t_a\partial_a Z(t) = N_a(\omega) Z(t)$, $a\in \arcs\cup\loops$. The existence of a solution of the normal equation stems from the general theory of exponential families if the transition's observed values are strictly positive. If it is not the case, we can find a border solution using the following characterization of transition's counts. 
 \begin{proposition}\label{prop:grande}
 An integer matrix $N \in \integers_{\ge}^{V\times V}$ is the transition count of a trajectory $\omega$ if, and only if, it is connected and $\sum_w N_{v\to w} - \sum_w N_{w\to w}$ is either 0 for all $v \in V$ or is $+1$ for a vertex $v_0$, $-1$ for a different vertex $v_n$, and zero for all $v\ne v_0,v_n$.
 \end{proposition}

 An observed transition's count $N(\omega)$ is connected and defines a sub-graph of the model graph by the positivity condition $N_a(\omega)>0$. It follows that the normal equations have a solution in a submodel.

 We have defined a toric ideal associated with the unnormalized densities on $\Omega_n$. Now we want to consider all trajectories, i.e. trajectories of any length $n$. We need first to fix a more precise language for closed trajectories.

 A \emph{closed trajectory} $\omega = v_0 v_1 \cdots v_{n-1}v_0$ is any trajectory going from an initial $v_0$ back to $v_0$; $r\omega = v_0 v_{n-1} \cdots v_1 v_0$ is the reversed closed trajectory. If we do not distinguish any initial vertex, the equivalence class of closed trajectories is called a \emph{trail}. A closed trajectory is \emph{elementary} if it has no proper closed sub-trajectory, i.e. if does not meet twice the same vertex except the initial one $v_0$. The trail of an elementary closed path is a \emph{cycle}. The set of cycles $\cycles$ is finite. A trajectory $\omega = \omega_0 \omega_1 \cdots \omega_{n-1} \omega_n$ is elementary if does not contain any cycle.

 Given a transition matrix $P$ and an initial probability $\pi_0$, we compute the probability of any trajectory as $\probof{\omega} = \pi_0(\omega_0) \left(\prod_{v,w\in V} P_{v\to w}^{N_{v\to w}(\omega)}\right)$. The factor $\left(\prod_{v,w\in V} P_{v\to w}^{N_{v\to w}(\omega)}\right)$ does not depend on the initial point $\omega_0$ if the trajectory is closed; in fact in such a case the matrix $N(\omega)$ is a function of the trail only.
 \begin{definition} Consider the ring $k[t_0; t_v,v\in V; t_a,a\in\arcs\cup\loops]$.
   \begin{enumerate} 
   \item The \emph{Markov monomial ideal} is the monomial ideal generated by the monomials $t_0 \prod_{v\in V} t_v^{(X_0(\omega)=v)} \prod_{a\in \mathcal A\cup \loops} t_a^{N_a(\omega)}$, where $\omega$ is a trajectory, $\omega\in \Omega = \cup_n \Omega_n$.
   \item The \emph{stationary Markov ideal} is the ideal generated by the Markov monomial ideal and by the equations of stationarity $\sum_{w\in\out(v)} t_v t_{v\to w} = t_w$, $w\in V$.
 \item The \emph{ideal of closed trajectories} is the monomial ideal generated by the monomials  $\prod_{a\in \arcs\cup\loops} t_{a}^{N_{a}(\omega)}$, $\omega$ closed, in the ring $k[t_a \colon a \in \arcs\cup\loops]$.   
   \end{enumerate}
 \end{definition}
 \begin{proposition}\label{prop:monimialcycles}
   \begin{enumerate}
   \item For each trajectory $\omega$ there exist an elementary sub-trajectory $\omega_{\text{e}}$, possibly empty, and nonnegative integers $\lambda(c)$, $c\in\cycles$, such that $N(\omega) = N(\omega_{\text{e}}) + \sum_{c\in \cycles} \lambda(c) N(c)$. All matrices $N(\omega_{\text{e}})$ and $N(c)$, $c\in\cycles$ are boolean. If $\omega_{\text{e}}$ is not empty, it has the same initial and final point as $\omega$.
 \item \label{item:monimialcycles2} The monomial ideal of closed trajectories is generated by the cycles. The monomials associated to a cycle are square-free.
 \item The Markov monomial ideal is generated by the cycles and the elementary trajectories.
 \end{enumerate}
 \end{proposition}
 \begin{proof}
   \begin{enumerate}
   \item Let $\omega=v_0\cdots v_n$ be a trajectory and consider the first closed trajectory encountered, $v_0\cdots v_h \cdots v_k(=v_h) v_{k+1} \cdots v_n$, if any. Then $c_1=v_h v_{h+1} \cdots v_k(=v_h)$ is an elementary closed trajectory and $\omega_{\text{r}} = v_0 \cdots v_h v_{k+1} \cdots v_n$ is either empty or a trajectory. Hence $N(\omega) = N(\omega_{\text{r}}) + N(c_1)$ and the iterations stops after a finite number of steps.
   \item If $\omega$ is closed, then $N(\omega) = \sum_{c\in\cycles} \lambda(c) N(c)$, hence
 $\prod_{a\in\arcs\cup\loops} t_a^{N_a(\omega)} = \prod_{c\in\cycles} \left(\prod_{a\in\arcs\cup\loops} t_a^{N_a(c)}\right)^{\lambda(c)}$. On a closed trajectory $\omega$, the matrix of transition counts
 \begin{equation*}
   [N_{v,w}(\omega)]_{v,w\in V} = \sum_{k=1}^n (X_{k-1}(\omega)=v,X_k(\omega)=w)
 \end{equation*}
 has row sums equal to column sums, i.e. there are as many \emph{ins} as \emph{outs} at each vertex, see Prop. \ref{prop:grande}. 
 \item If the trajectory is closed, then we can apply the previous item. Otherwise $\omega$ and $\omega_e$ start at the same vertex and
   \begin{equation*}
     t_0 \prod_{v\in V} t_v^{(X_0(\omega)=v)} \prod_{a\in \mathcal A\cup \loops} t_a^{N_a(\omega)} = t_0 \prod_{v\in V} t_v^{(X_0(\omega_{\text{e}})=v)} \prod_{a\in \mathcal A\cup \loops} t_a^{N_a(\omega_{\text{e}})} \ .
   \end{equation*}

 \end{enumerate}
 \end{proof}

 In a sense, the joint distribution of each trajectory is characterized by the Markov monomial ideal.
 \begin{proposition}
 Let $P$ be an irreducible Markov matrix.  The distribution of the stationary Markov chain is a function of transition's monomials of the cycles.
 \end{proposition}
 \begin{proof}
Consider $v \ne w$ with $P_{v\to w} > 0$ and define the sets $\Omega_k(w\to v)$ of the trajectories from $w$ to $v$ of length $k$, $k=1,2,\dots$. For each $\omega\in\Omega_k(w\to v)$, the trajectory $(v\to w)\omega=v(\omega_0=w)\cdots (\omega_k=v)$ is closed. Because of the Markov property and irreducibility, we have the Cesaro limit $\pi(v) = \lim_{k\to\infty} P_{w\to v}^{(k)} > 0$, so that
 \begin{multline*}
 \lim_{k\to\infty} \sum_{\omega\in\Omega(k,w\to v)} \probof{(v\to w)\omega}/\pi(v) = \\ \lim_{k\to\infty}  P_{v\to w} \sum_{\omega\in\Omega(k,w\to v)} \probof{\omega|\omega_0=w} = \\ P_{v \to w} \lim_{k\to\infty} P^{(k)}_{w\to v} = \pi(v)P_{v \to w} \ .   
 \end{multline*}
As $\probof{(v\to w)\omega}/\pi(v)$ is a product of cycle's monomials because of Prop. \ref{prop:monimialcycles}\eqref{item:monimialcycles2}, all values $P(v,w) = \pi(v)P_{v \to w}$ of the 2-step joint distribution depend on the values of the cycle's monomials.
\end{proof}

We note that the values of cycle's monomials are dependent. For example, in the complete graph with 4 vertices there are 20 cycles (including the 4 loops), while the number of degrees of freedom for a generic transition probability on 4 points is 12.

 \subsection*{Notes} The name TMC was used first by Pachter and Sturmfels\cite[Ch. 1 Statistics]{pachter|sturmfels:2005}. See also Hare and Takemura \cite{takemura|hara:1004.3599T,hara|takemura:1005.1717H} for slightly different definitions. The representation of processes whose sufficient statistics are transition's count is a version of de Finetti exchangeability, see Freedman\cite{freedman:1962invariants}. Prop. \ref{prop:grande} is proved Grande\cite{grande:2011tesi}. A TMC is a particular case of a graphical model in the sense of Lauritzen\cite{lauritzen:1996}. A related issue is the characterization of the distribution of stationary Markov chains via a mixture on monomials on cycles that was obtained by McQueen\cite{macqueen:1981}, cfr. Kalpazidou\cite{kalpazidou:2006}.
 \section{Reversible Markov chains}
 A transition matrix $P_{v\to w}$, $v, w \in V$, satisfies the \emph{detailed balance} (DB) condition if $\kappa(v)P_{v\to w} = \kappa(w)P_{w\to v}$, $v, w \in V$ for some strictly positive $\kappa(v) > 0$, $v \in V$. As a consequence, $\pi(v) \propto \kappa(v)$ is an invariant probability and the stationary Markov chain $(X_n)_{n\in\integers_{\ge}}$ has \emph{reversible} two-step joint distribution $\probof{X_n = v, X_{n+1} = w} = \probof{X_n = w,X_{n+1}=v}$, $v, w \in V$, $n \ge 0$. The distribution of the Markov chain is uniquely parameterized by its symmetric two-step joint distribution. 

 The DB assumption is trivially satisfied for $v=w$ and moreover $P_{v\to w} > 0$ if, and only if, $P_{w\to v} > 0$. Given an undirected connected graph $\mathcal G = (V,\edges)$ with no loops and the directed graph $(V,\arcs)$ whose arcs are the two directions of each edge, we consider here all transition probabilities such that $P_{v\to w} = 0$, $v\ne w$, if, and only if, $\overline{vw} \notin \edges$. Let $\loops$ denote loops with positive transition $P_{v\to v}$ and $\pi$ the invariant probability.

 For each trajectory $\omega=v_0\cdots v_n$ in the graph $\mathcal G$ let $r\omega=v_n\cdots v_0$ be the \emph{reversed trajectory}. The reversed probability is $\probat r \omega = \probof{r\omega}$. From \eref{eq:MCjoint} we compute the likelihood
 \begin{align*}
 \frac{\probof \omega}{\probat r \omega} &= \frac{\prod_{v\in V} \pi(v)^{(X_0(\omega)=v)} \prod_{a\in \mathcal A\cup \loops} P_a^{N_a(\omega)}}{\prod_{v\in V} \pi(v)^{(X_0(r\omega)=v)} \prod_{a\in \mathcal A\cup \loops} P_a^{N_a (r\omega)}} \\ &= \frac{\prod_{v\in V} \pi(v)^{(X_0(\omega)=v)} \prod_{a\in \mathcal A} P_a^{N_a(\omega)}}{\prod_{v\in V} \pi(v)^{(X_n(\omega)=v)} \prod_{a\in \mathcal A} P_{ra}^{N_a(\omega)}}  \\ &= \prod_{v\in V} \pi(v)^{(X_0(\omega)=v)-(X_n(\omega)=v)} \prod_{a\in \mathcal A} \left(\frac{P_a}{P_{ra}}\right)^{N_a(\omega)} \ ,
 \end{align*}
 the log-likelihood
 \begin{multline*}
  \logof{\frac{\Prob}{\Prob_r}} = \\ \sum_{v\in V} \logof{\pi(v)} ((X_0=v)-(X_n=v)) + \sum_{a \in \mathcal A} \logof{\frac{P_a}{P_{ra}}} N_a \ .
 \end{multline*}
 because of the stationarity of $\pi$, the divergence is
 \begin{multline*}
 \divergence{\Prob}{\Prob_r} = \\ \sum_{v\in V} \logof{\pi(v)} \expectat{\Prob}{(X_0(\omega)=v)-(X_n(\omega)=v)} + \sum_{a\in \mathcal A} \logof{\frac{P_a}{P_{ra}}} \expectat{\Prob}{N_a} \\ = n \sum_{(v\to w) \in \mathcal A} \logof{\frac{P_{v\to w}}{P_{w\to v}}} \pi(v) P_{v\to w} \\ = n \sum_{v, w \in V} \logof{\frac{P(v,w)}{P(w,v)}} P(v,w) \ ,
 \end{multline*}
 where $P(v,w)$ is the two-step joint distribution. As the divergence is zero if, and only if, the probabilities are equal, the DB condition is equivalent to $\Prob=\Prob_r$. The last statement could be easily derived otherwise, but the computation of the divergence has an independent interest, e.g. it shows the linear dependence on $n$ of the divergence.

 Let $\omega$ a \emph{closed trajectory} and let $r\omega$ its reversed trajectory. In the previous section we have shown that the distribution of the Markov chain is uniquely characterized by the initial probability, the loop transitions $P_{v\to v}$, $v \in V$, and the monomials $\prod_{a\in\arcs} P_a^{N_a(\omega)} = P^{\omega}$, for each closed and loop-free $\omega$. 

 In the case of a reversible chain, these monomials are invariant under the reversion.
 \begin{proposition}[Kolmogorov]
 Let the Markov chain $\left(X_n\right)_{n\in \integers_{\ge}}$ have its transitions supported by the connected graph $\mathcal G$. The MC is reversible if, and only if, $P^\omega = P^{r\omega}$ for all closed trajectory $\omega$. 
 \end{proposition}
 This suggests the following definition.
 \begin{definition}The \emph{Kolmogorov's ideal} or \emph{K-ideal} of the graph $\mathcal G$ is the ideal generated by the binomials $P^\omega-P^{r\omega}$, where $\omega$ is a closed trajectory.
 \end{definition}

 The generation of the ideal of closed trajectory by the cycles of Prop. \ref{prop:monimialcycles} together with the sygyzy characterization of Gr\"obnber basis on an ideal lead to the following proposition.
 \begin{proposition}
   \begin{enumerate}
   \item The K-ideal is generated by the set of binomials $P^\omega - P^{r\omega}$, where $\omega$ is cycle.
 \item The binomials $P^\omega-P^{r\omega}$, where $\omega$ is any cycle, form a \emph{reduced universal Gr\"obner basis} of the K-ideal.  
 \end{enumerate}
 \end{proposition}

 In the previous sections we have constructed a process leading from a monomial parameterization of a statistical model to a binomial basis of its toric ideal. Here the process is reversed, in that we are given a binomial ideal and would like to show it is, in fact, a toric ideal. This program will eventually produce a parameterization of reversible Markov transitions in monomial form. We observe that the variety of the K-ideal is not satisfied by probabilities, but by transition probabilities. We had a similar issue when describing the difference between TMCs and MCs.

 The proof that the K-ideal is a toric ideal is based on standard construct of graph theory that we are going to review now. Let $\mathcal C$ be the set of cycles of $\arcs$.  For each cycle $\omega \in \mathcal C$ we define the \emph{cycle vector} of $\omega$ to be $z(\omega) = (z_a(\omega): a\in\arcs)$, where
 \begin{equation*}
   z_a(\omega) = \begin{cases}
 +1 & \text{if $a$ is an arc of $\omega$}, \\ -1 & \text{if $r(a)$ is an arc of $\omega$},\\
  0 & \text{otherwise.}
   \end{cases}
 \end{equation*}
 The \emph{cycle space} is the vector space generated in $\reals^{\arcs}$ by the cycle vectors.

 For each proper subset $B$ of the set of vertices, $\emptyset \ne B \subsetneq V$, say $B\in\cuts$, we define the \emph{cocycle vector} of $B$ to be $u(B) = (u_a(B): a \in \arcs)$, with
 \begin{equation*}
   u_a(B) = \begin{cases}
 +1 & \text{if $a$ exits from $B$,} \\
 -1 & \text{if $a$ enters into $B$,} \\
 0 & \text{otherwise.}
   \end{cases}
 \end{equation*}

 The cocycle vectors generate the \emph{cocycle space}. The cycle space and the cocycle space orthogonally split $\reals^\arcs$. We denote by $\integers(\arcs)$ the integer-valued cycle vectors. We will see below how this integer vectors are related to transition's counts. 

 The model matrix of the toric model we are going to produce for the K-ideal is the matrix with row's indices in $\edges\cup\cuts$ and column's indexes in $\arcs$. The element in position $(e,a)$ of the $(\edges\times\arcs)$-block is one if, and only if, the arc $a$ belongs to the edge $e$. The element in position $(B,a)$ of the $(\cuts\times\arcs)$-block is $u_B(a)$. We call this model matrix the \emph{cocycle matrix}. It follows that the cocycle space is the kernel of the cocycle matrix and $\integers(\arcs)$ is its lattice.

 The following definition provides a generalization of the Hilbert basis we already used in our discussion of the border of an $A$-model. It is needed below to provide a more precise version of the cycle decomposition of a closed trajectory. 

 \begin{definition}
   \begin{enumerate}
   \item Given two integer vectors $z_1, z_2 \in \integers^\arcs$, we say $z_1$ is \emph{conformal} to $z_2$, $z_1 \sqsubseteq z_2$, if the component-wise product is nonnegative and $|z_1| \le |z_2|$ component-wise, i.e. $z_{1,a}z_{2,a} \ge 0$ and $|z_{1,a}| \le |z_{2,a}|$ for all $a \in \arcs$. 
 \item A \emph{Graver basis} of $\integers(\arcs)$ is the set of the minimal elements with respect to the conformity partial order $\sqsubseteq$.
   \end{enumerate}
 \end{definition}

 \begin{proposition}
 \begin{enumerate}
 \item
 For each integer cycle vector  $z \in \integers(\arcs)$, $z = \sum_{\omega \in \mathcal C} \lambda(\omega) z(\omega)$, there exist cycles $\omega_1, \dots, \omega_n \in \mathcal C$ and positive integers $\alpha(\omega_1),\dots,\alpha(\omega_n)$, such that $z^+\ge z^+(\omega_i)$, $z^-\ge z^-(\omega_i)$, $i=1,\dots,n$ and $z = \sum_{i=1}^n \alpha(\omega_i) z(\omega_i).$
 \item The set $\left\{z(\omega) \colon \omega \in \mathcal C \right\}$  is a \emph{Graver basis} of  $\integers(\arcs)$.
 \end{enumerate}
 \end{proposition}

 From the previous proposition follows the result on the K-ideal.

 \begin{proposition}
 The K-ideal is the toric ideal of the cocycle matrix. 
 \end{proposition}
 In fact, the binomials $P^\omega-P^{r\omega}$, $\omega\in\cycles$, form a Graver basis of the K-ideal. The cocycle matrix has negative entries and it could easily modified to a matrix with nonnegative entries as we have done in the discussion of the Ising model.

 The previous algebraic statement is rephrased in statistical terms as follows. 

 \begin{enumerate}
 \item The \emph{strictly positive reversible transition probabilities} on $(V,\arcs)$ are given by
 \begin{align*}
   P_{v\to w} &= s(v,w) \prod_B t_B^{u_{v\to w}(B)}\\  &= s(v,w)
    \prod_{B \colon v\in B, w \notin B } t_B \ \prod_{B \colon w\in B, v \notin B} t_B^{-1} \ ,
 \end{align*}
 where $s(v,w) = s(w,v) > 0$, $t_B >0$.
 \item The first set of parameters, $s(v,w)$, is a function of the edge $\overline{vw}\in\edges$.
 \item The second set of parameters, $t_B$, $B\in\cuts$, represents the deviation from symmetry. It is not identifiable because the full set of cocycle vectors $u_B$, $B\in\cuts$, is not linearly independent.
 \item The parametrization can be used to derive an explicit form of the invariant probability.
 \end{enumerate}

 The following proposition is a summary of all results.

 \begin{proposition}
 Consider the strictly non-zero points on the K-variety.
 \begin{enumerate}
 \item
 The symmetric parameters $s(e)$, $e\in\edges$, are uniquely determined. The parameters $t_{B}$, $B\in\cuts$ are confounded by the $(\cuts\times\arcs)$-block of the cocycle matrix.
 \item
 An identifiable parametrization is obtained by taking a subset of parameters corresponding to linearly independent rows, denoted by $t_{B}$, $B\in\mathcal T$, $\mathcal T \subset \mathcal S$:
 \begin{equation*}
   P_{v\to w} =  s(v,w) \prod_{B\in\cuts \colon v\in B, w \notin B } t_B \
   \prod_{B\in\cuts \colon w\in B, v \notin B} t_B^{-1} \ .
 \end{equation*}
 \item
 The detailed balance equations, $ \kappa(v) P_{v\to w}=\kappa(w) P_{w\to v} $, are verified if, and only if,
 \begin{equation*}
 \kappa(v) \propto \prod_{B  \colon v\in S} t_B^{-2} \ .
 \end{equation*}
 \end{enumerate}
 \end{proposition}

 It is possible to give an algebraic form of the original Kolmogorov statement on the equivalence of detailed balance with equality of transitions on closed trajectories in the form of a statement on elimination ideals.

 \begin{definition}
 The \emph{detailed balance ideal} is the ideal 
 \begin{equation*}
   \idealof{\prod_{v\in V} \kappa(v) - 1, \kappa(v)P_{v\to w}- \kappa(w)P_{v\to w}, \  (v\to w)\in \arcs}
 \end{equation*}
 in the ring $\rationals[\kappa(v), v\in V; P_{v\to w}, (v\to w)\in \arcs]$.
 \end{definition}

 \begin{proposition}
 \begin{enumerate}
 \item
 The matrix $\left[P_{v\to w}\right]_{v\to w \in \arcs}$ is a point of the variety of the K-ideal  if and only if there
 exists $\kappa=\left(\kappa(v) \colon v\in V \right)$ such that $\left(\kappa,P \right)$ belongs to the variety of the
 detailed balance ideal.
 \item
 The detailed balance ideal is a toric ideal.
 \item
 The K-ideal is the $\kappa$-elimination ideal of the detailed balance ideal.
 \end{enumerate}
 \end{proposition}

 By combining the monomial representation of the transitions and the monomial representation of the invariant probability, we obtain a classical parameterization of reversible transitions in the form
 \begin{equation*}
  P_{v\to w} =  s(v,w) \kappa(w)^{1/2} \kappa(v)^{-1/2} \ ,
 \end{equation*}
 together with the constrain
 \begin{equation*}
   \kappa(v)^{1/2}\ge \sum_{w \ne v} s(u,w) \kappa(w)^{-1/2} \ .
 \end{equation*}
 This form is well known in the literature on the \emph{Hastings-Metropolis} simulation algorithm, where we are given an unnormalized positive probability $\kappa$ and a transition $Q_{v\to w} > 0$ if $(v\to w) \in \arcs$. We are required to produce a new transition $P_{v\to w} = Q_{v\to w} \alpha(v,w)$ such that $P$ is reversible with invariant probability $\kappa$ and $0 < \alpha(v,w) \le 1$. We have
 \begin{equation*}
 Q_{v\to w} \alpha(v,w) = s(v,w) \kappa(w)^{1/2} \kappa(v)^{-1/2}\end{equation*}
 and moreover we want
 \begin{equation*}
   \alpha(v,w) = \frac{s(v,w)\kappa(w)^{1/2}}{Q_{v\to w}\kappa(v)^{1/2}} \le 1.
 \end{equation*}

 \begin{proposition}
 Let $Q$ be a probability on $V\times V$, strictly positive on $\edges$, and let $\pi(x) = \sum_y Q(x,y)$. If $f: ]0,1[\times]0,1[ \to ]0,1[$ is a symmetric function such that $f(u,v) \le u \wedge v$ then
     \begin{align*}
       P(x,y) =
       \begin{cases}
         f(Q(x,y),Q(y,x)), & \set{x,y}\in\edges, \\
         \pi(x) - \sum_{y\colon y \ne x} P(x,y), & x=y, \\ 0 & \text{otherwise} \ ,
       \end{cases}
     \end{align*}
 is a 2-reversible probability on $\edges$ such that $\pi(x) = \sum_y P(x,y)$, positive if $Q$ is positive.
 \end{proposition}

 The proposition applies to various cases of interes, e.g. $f(u,v) = u \wedge v$, $f(u,v) = uv/(u+v)$, $f(u,v) = uv$. 

 \subsection*{Notes}
 A recent exposition of the theory of reversible MCs appears in the lecture notes by Aldous and Fill\cite{aldous|fill:2002ch3}. We use standard results in graph theory, see e.g. the monograph by Bollobas\cite{bollobas:1998}. Here we mainly follow our paper\cite{pistone|rogantin:1007.4282v2}. The application to simulation is discussed e.g. in the monograph by Liu\cite{liu:2008strategies}.

 \section*{Discussion}
 The basic notions in Probability and Mathematical Statistics are those of independence and conditioning, which in turn are expressed by product of probability and conditional probability. It is not by chance that almost all classical statistical models have a product form and that the logarithmic transformations are a key ingredient of computation in Statistics. Modern Combinatorial and Computational Commutative Algebra have provided a unifying framework for the diverse instances of such basic structures. In particular, the notion of toric ideal captures the essentials of what we have called Lattice Exponential Families, i.e. the discrete case that is rife in Applied Statistics and Statistical Physics. We have presented an overview of the algebraic theory of statistical models that fall under the scheme, together with an algebraic discussion of their limit cases and their differentiation. The application to Markov chains is an expansion of these ideas to objects that do not belong to a probability simplex. This is actually a special case of a more general interesting theory, namely Bayes networks. It is likely to expect applications to bayesian statistics, a topic we have not touched upon.

 We thank the organizers of the Osaka conference, especially professor Takayuki Hibi for providing an ideal place to discuss these, and related, ideas. We thank Paolo Baldi, Francesco Grande, Luigi Malag\`o, Fabio Rapallo for useful discussions while this paper was in preparation.

 \bibliographystyle{ws-procs9x6}

\end{document}